\documentclass[twoside=semi,leqno,abstracton
]{scrartcl}
\pdfoutput=1
\setkomafont{subtitle}{\itshape}
\setkomafont{disposition}{\normalfont\normalcolor\bfseries}
\setkomafont{dictum}{\scshape}
\setkomafont{paragraph}{\bfseries}
\setkomafont{pagehead}{\normalfont\upshape\small}
\setkomafont{pagenumber}{\normalfont\upshape\normalsize}
\usepackage[arxiv%
]{MyKoma}
\usepackage{dsfont}

\makeatletter

\def\dom{\mathrm{dom}}

\makeatother

\begin{document}

\title{On the Lebesgue Property of Monotone Convex Functions}
\runtitle{Lebesgue Property of Monotone Convex Functions}
\author{Keita Owari}%
\runauthor{K. Owari}

 
\address{Graduate School of Economics, The University of Tokyo\newline
  7-3-1 Hongo, Bunkyo-ku, Tokyo 113-0033, Japan}

\email{owari@e.u-tokyo.ac.jp}
\keyAMS{46E30, 47H07, 46N10, 91G80, 91B30 }%
\keyJEL{C02, C60}%

\keyWords{monotone convex functions, Lebesgue property,
  order-continuity, perturbed James's theorem, convex risk measures }

\FrstVer{4 Jun. 2013, Accepted: 16 Nov. 2013}

\ToAppear{Math. Financ. Econ.}
\DOI{10.1007/s11579-013-0111-z}


\abstract{%
  The Lebesgue property (order-continuity) of a monotone convex
  function on a solid vector space of measurable functions is
  characterized in terms of (1) the weak inf-compactness of the
  conjugate function on the order-continuous dual space, (2) the
  attainment of the supremum in the dual representation by
  order-continuous linear functionals. This generalizes and unifies
  several recent results obtained in the context of convex risk
  measures.
}

\maketitle

\section{Introduction and the Statement of the Result}
\label{sec:MonConvSolid}

Many problems in mathematical finance and economics involve some
monotone convex functions of measurable functions, and their
regularity with respect to the natural order structure often plays a
key role. In this short note, we characterize the \emph{Lebesgue
  property} (order-continuity) of such functions on solid spaces of
measurable functions, in terms of the conjugate defined on the
order-continuous dual space, which unifies the recent studies in the
context of convex risk measures (\citep{MR2011534,MR2648597},
\citep{jouini06:_law_fatou}, \citep{orihuela_ruiz12:_lebes_orlic},
\citep{kratschmer07}).

We use the probabilistic notation.  Throughout the paper,
$(\Omega,\FC,\PB)$ denotes a fixed probability space.
$L^0:=L^0(\Omega,\FC,\PB)$ stands for the space of (equivalence
classes modulo $\PB$-almost sure (a.s.) equality of) \emph{finite}
measurable functions, and we write simply $L^p:=L^p(\Omega,\FC,\PB)$
for the measure $\PB$, while $L^p(\QB):=L^p(\Omega,\FC,\QB)$ for other
measures. With the a.s. pointwise order, $L^0$ is an order-complete
Riesz space with the \emph{countable-sup property} (see
\citep[][Ch.~8]{aliprantis_border06} for these terminologies). We fix
a \emph{solid} vector subspace (ideal) $\Xs$ of $L^0$, that is, a
vector subspace of $L^0$ such that $|X|\leq |Y|$ (a.s.) and $Y\in \Xs$
imply $X\in \Xs$. Then $\Xs$ is an order-complete Riesz space with the
countable-sup property on its own right.  All Orlicz spaces and their
Morse subspaces including $L^p$'s are solid in this sense.  We suppose
that $\Xs$ contains the constants, then $L^\infty\subset \Xs$ by the
solidness.

We work with the pairing $\langle \Xs,\Xs^\sim_n\rangle$ where
$\Xs^\sim_n$ is another solid space given by
\begin{equation}
  \label{eq:OderContiDual}
  \Xs^\sim_n=\{Z\in L^0:\, XZ\in L^1,\,\forall X\in \Xs\},
\end{equation}
with the
bilinear form $\langle X,Z\rangle=\EB[XZ]:=\int_\Omega XZd\PB$. In the
terminology of Riesz spaces, $\Xs^\sim_n$ is the
\emph{order-continuous dual} of $\Xs$ (with the identification of
$Z\in \Xs^\sim_n$ and the order-continuous linear functional $X\mapsto
\EB[XZ]$ on $\Xs$, see \citep[Section~112]{MR704021}).  Note that
$\Xs$ separates $\Xs^\sim_n$ as long as $\Xs$ contains the constants
as assumed so and then $\Xs^\sim_n\subset L^1$, while
$\Xs^\sim_n=\{0\}$ is possible in general.  $\Xs^\sim_n$ separates
$\Xs$ if and only if $\Xs\subset L^1(\QB)$ for a probability measure
$\QB$ \emph{equivalent to} $\PB$ ($\QB\sim\PB$). The pair $\langle
\Xs,\Xs^\sim_n\rangle$ is then in separating duality, thus the weak
topology $\sigma(\Xs,\Xs^\sim_n)$ is a locally convex Hausdorff
topology.

By a \emph{monotone} convex function on a solid space $\Xs\subset
L^0$, we mean a proper convex function $\varphi:\Xs\rightarrow
(-\infty,\infty]$ such that $\varphi(X)\leq \varphi(Y)$ whenever
$X\leq Y$ (a.s.), and let
\begin{equation}
  \label{eq:ConjPsiUsual}
  \varphi^*(Z):=\sup_{X\in \Xs}(\EB[XZ]-\varphi(X)),\,Z\in \Xs^\sim_n.
\end{equation}

The aim of this note is to prove the following
(cf. \citep{MR2011534,MR2648597}, \citep{jouini06:_law_fatou},
\citep{orihuela_ruiz12:_lebes_orlic}, \citep{kratschmer07}).

\begin{theorem}
  \label{thm:JSTGeneral2}
  Let $\Xs\subset L^0$ be a solid space containing the constants and
  contained in $ L^1(\QB)$ for some probability $\QB\sim\PB$, and
  $\varphi:\Xs\rightarrow \RB$ a finite-valued monotone convex
  function which is $\sigma(\Xs,\Xs^\sim_n)$-lower semicontinuous or
  equivalently
  \begin{equation}
    \label{eq:FatouEquality1}
    \varphi(X)=\sup_{Z\in\Xs^\sim_n}(\EB[XZ]-\varphi^*(Z)),\,\forall X\in \Xs.
  \end{equation}
  Then the following are equivalent:
  \begin{description}[(3)]
  \item[(1)] $\varphi$ has the Lebesgue property on $\Xs$, that is,
    \begin{equation}
      \label{eq:LebX2}
      \exists Y\in \Xs\text{ such that }
      |X_n|\leq |Y|,\,\forall n\text{ and }X_n\rightarrow X\text{ a.s. } 
      \Rightarrow\,\varphi(X)=\lim_n\varphi(X_n);
    \end{equation}
  \item[(2)] $\varphi^*$ is $\sigma(\Xs^\sim_n,\Xs)$-inf-compact,
that is, $\{Z\in\Xs^\sim_n:\,\varphi^*(Z)\leq c\}$ is
    $\sigma(\Xs^\sim_n,\Xs)$-compact for each $c>0$;
  \item[(3)] $\sup_{Z\in\Xs^\sim_n}(\EB[XZ]-\varphi^*(Z))$ is attained
    for all $X\in\Xs$ and in particular,
    \begin{equation}
  \label{eq:FatouFrechet}
  \varphi(X)
  =\max_{Z\in\Xs^\sim_n}(\EB[XZ]-\varphi^*(Z)),\,\forall X\in\Xs.
\end{equation}

\end{description}

\end{theorem}
\begin{remark}[Order-continuity]
  By the countable-sup property of $\Xs$ as an ideal of $L^0$, the
  Lebesgue property (\ref{eq:LebX2}) is equivalent to the generally
  stronger \emph{order-continuity}:
  $\varphi(X)=\linebreak\lim_\alpha\varphi(X_\alpha)$ if a \emph{net}
  $(X_\alpha)_\alpha\subset\Xs$ converges in order to $X$, i.e., if
  there is a decreasing net (with the same index set)
  $(\xi_\alpha)_\alpha$ with $|X-X_\alpha|\leq \xi_\alpha\downarrow 0$
  in $\Xs$.
\end{remark}
A proof will be given in Section~\ref{sec:ProofMain}. Here we collect
some remarks and consequences. We first emphasize that all Orlicz
spaces as well as their Morse subspaces including $L^p$ with $1\leq
p\leq \infty$ are covered by Theorem~\ref{thm:JSTGeneral2}. Also, any
solid space $\Xs\subset L^0$ which admits a \emph{finite} monotone
convex function is contained in $L^1(\QB)$ with some $\QB\ll \PB$,
thus only the equivalence $\QB\sim\PB$ does really matter in the
assumption regarding $\QB$.

We can relate the Lebesgue property to some other common regularity
properties:
\begin{corollary}
  \label{cor:KnownEquivalences}
  In the situation of Theorem~\ref{thm:JSTGeneral2}, the equivalent
  conditions (1) -- (3) are further equivalent to any of the
  following:
  \begin{enumerate}\setcounter{enumi}{3}
  \item $\varphi$ is
    $\sigma(\Xs,\Xs^\sim_n)$-subdifferentiable, i.e.,
    \begin{equation}
      \label{eq:subdiff2}
      \forall X\in\Xs,\,\exists Z\in \Xs^\sim_n\text{ such that }
      \EB[XZ]-\varphi(X)\geq\EB[YZ]-\varphi(Y),\,\forall Y\in\Xs;
    \end{equation}
  \item $\varphi$ is continuous for the Mackey topology
    $\tau(\Xs,\Xs^\sim_n)$.

  \end{enumerate}

\end{corollary}
\begin{proof}
  (3) $\Leftrightarrow$ (4) is just a paraphrasing since $\bar
  Z\in\Xs^\sim_n$ maximizes $Z\mapsto \EB[XZ]-\varphi^*(Z)$ on
  $\Xs^\sim_n$ if and only if $\EB[X\bar Z]-\varphi(X)
  =\varphi^*(Z_X)=\sup_{Y\in\Xs}(\EB[Y\bar Z]-\varphi(Y))$ by the
  definition of $\varphi^*$. Also, (2) $\Leftrightarrow$ (5) is true
  for any finite convex function on a vector space forming a dual pair
  with another vector space, which is lower semicontinuous w.r.t. a
  topology consistent with the duality (e.g.  \citep{MR0160093},
  Propositions~1 and 2). 
\end{proof}

\begin{remark}[On subdifferentiability]
  The message of (4) is that the subdifferential of $\varphi$ contains
  a $\sigma$-additive element (rather than it is non-empty).  Consider
  the case where $\Xs$ is given a completely metrizable topology
  $\tau$ for which $(\Xs,\tau)$ is a \emph{locally convex Fréchet
    lattice} (w.r.t. the same a.s. pointwise order). Then the
  $\tau$-dual of $\Xs$ has the direct sum decomposition
  $\Xs^*=\Xs^\sim_n\oplus\Xs^\sim_s$ where $\Xs^\sim_s$ is the set of
  \emph{singular} linear functionals on $\Xs$ (see
  \citep[Ch.~8]{aliprantis_border06} for detail). In this case, an
  \emph{extended Namioka-Klee theorem} of \citep{MR2648595} asserts
  that any \emph{finite} monotone convex function $\varphi$ on $\Xs$
  is automatically $\tau$-continuous and subdifferentiable in the
  usual sense: for any $X\in \Xs$, there exists a $\nu\in \Xs^*$ such
  that
  \begin{equation}
    \label{eq:Subdiff3}
    \nu(X)-\varphi(X)\geq\nu(Y)-\varphi(Y),\,\forall Y\in\Xs.
  \end{equation}
  See \citep[Theorem~1]{MR2648595} (see also
  \citep[][Proposition~3.1]{MR2254417}). Thus when $\Xs$ is a locally
  convex Fréchet lattice, Corollary~\ref{cor:KnownEquivalences} tells
  us that such $\nu$ can be taken to be order continuous
  ($\Leftrightarrow$ $\sigma$-additive in our setting) if and only if
  $\varphi$ has the Lebesgue property (the sufficiency is already
  obtained by \citep[Lemma~7]{MR2648595}).
\end{remark}

As the Lebesgue property (\ref{eq:LebX2}) is the continuity w.r.t. the
order structure, the order \emph{lower semicontinuity} is often called
the \emph{Fatou property} and characterized by
\begin{equation}
  \label{eq:FatouX2}
  \exists Y\in \Xs\text{ such that }
  |X_n|\leq |Y|,\,\forall n\text{ and }X_n\rightarrow X\text{ a.s. }
  \Rightarrow\,\varphi(X)\leq\liminf_n\varphi(X_n).
\end{equation}
Clearly, (\ref{eq:LebX2}) implies (\ref{eq:FatouX2}), and the latter
is closely related to the $\sigma(\Xs,\Xs^\sim_n)$-lower
semicontinuity ($\Leftrightarrow$ (\ref{eq:FatouEquality1})) assumed
in Theorem~\ref{thm:JSTGeneral2}. In fact, (\ref{eq:FatouEquality1})
$\Rightarrow$ (\ref{eq:FatouX2}) is always true (see
\citep[][Proposition~1]{MR2648595}).  If the converse was also true,
Theorem~\ref{thm:JSTGeneral2} would provide us an even nicer
interpretation of the Lebesgue property as the Fatou property (easy to
check) plus ``something extra'', with the ``extra'' being precisely
specified. See
\citep[Theorem~3.9]{owari13:_maxim_lebes_exten_monot_convex_funct} for
more discussion, where a characterization of Lebesgue property in the
form of Theorem~\ref{thm:JSTGeneral2} is obtained under solely the
Fatou property as the a priori assumption, but with the conjugate
$(\varphi|_{L^\infty})^*$ of the restriction to $L^\infty$ of
$\varphi$ instead of $\varphi^*$.  The implication (\ref{eq:FatouX2})
$\Rightarrow$ (\ref{eq:FatouEquality1}) for proper convex functions is
indeed true for some good spaces $\Xs$, but its validity in the
generality of this paper is still open (to us).

\begin{remark}
  The above question regarding the Fatou and
  $\sigma(\Xs,\Xs^\sim_n)$-lower semicontinuity is equivalent to
  asking if all \emph{order closed convex} sets in $\Xs$ are
  $\sigma(\Xs,\Xs^\sim_n)$-closed.  In \cite[Lemma~6 and
  Corollary~4]{MR2648595}, it is claimed that this is true whenever
  $\Xs$ is (lattice homomorphic to) an ideal of $L^1$, but the proof
  given there has an error. Adapting to our notation, they argued that
  as $\Xs$ being an ideal of $L^1$, $\Xs^\sim_n$ contains $L^\infty$,
  so the $\sigma(\Xs,\Xs^\sim_n)$-convergence of a net
  $(X_\alpha)_\alpha$ in $\Xs$ to $X$ implies the
  $\sigma(L^1,L^\infty)$-convergence to the same limit \emph{as a net
    in $L^1$}, and consequently there exist a sequence of indices
  $(\alpha_n)_n$ as well as $Y_n \in
  \mathrm{conv}(X_{\alpha_n},X_{\alpha_{n+1}},\ldots)$ such that
  $Y_n\rightarrow X$ \emph{in order in $L^1$}. From this they
  concluded that $Y_n\rightarrow X$ in order in $\Xs$, which would
  prove the desired implication for all \emph{convex sets} in
  $\Xs$. The error lies in the last part. More specifically, the order
  convergence of a sequence in an ideal $\Xs$ of $L^0$ is equivalent
  to the \emph{dominated a.s. convergence} (i.e., $X_n\rightarrow X$
  a.s. and $\exists Z\in \Xs_+$ with $|X_n|\leq Z$ a.s. for all $n$);
  the a.s. convergence is common to all ideals of $L^0$, while being
  dominated \emph{by an element of $\Xs$} is specific to each $\Xs$,
  so $Y_n\rightarrow X$ in order as a sequence in $L^1$ need not imply
  the order convergence in $\Xs$.

\end{remark}

\subsection{Examples and Related Literature}
\label{sec:Comments}

A monotone \emph{decreasing} convex function $\rho:\Xs\rightarrow
(-\infty,\infty]$ is called a \emph{convex risk measure} if it
satisfies $\rho(X+c)=\rho(X)-c$ for all constants $c$ (see
\citep{MR2779313} for financial motivation and use of this
notion). Making a change of sign, $\varphi(X)=\rho(-X)$ is a monotone
(increasing) convex function with
\begin{equation}
  \label{eq:CashAdd}
  \varphi(X+c)=\varphi(X)+c,\,X\in \Xs,\,c\in\RB.
\end{equation}
For this type functions, which we call \emph{convex risk functions},
\citep[Theorem~4.1]{kratschmer07} obtained a similar result regarding
the equivalence between (1) and (3), and the characterization in the
form of Theorem~\ref{thm:JSTGeneral2} including the weak compactness
(2) has been studied for some special solid spaces $\Xs$ as briefly
reviewed below. 




\subsubsection{$L^\infty$}
\label{sec:Linfty}

When $\Xs=L^\infty$, then $(L^\infty)^\sim_n=L^1$, thus
$\sigma(L^\infty,(L^\infty)^\sim_n)=\sigma(L^\infty,L^1)$ is the weak*
topology. Then on the one hand, a convex set $C\subset L^\infty$ is
$\sigma(L^\infty,L^1)$-closed if and only if $\{X\in C:\,
\|X\|_\infty\leq a\}$ is $L^0$-closed for all $a>0$ by the
Krein-Šmulian theorem, and on the other hand, the Fatou property
(\ref{eq:FatouX2}) is equivalent to
\begin{equation}
  \label{eq:FatouLinfty}
  \sup_n    \|X_n\|_\infty<\infty,\,X_n\rightarrow X\text{ a.s. }
  \Rightarrow\,  \varphi(X)\leq\liminf_n\varphi(X_n).
\end{equation}
Consequently, we have (\ref{eq:FatouEquality1}) $\Leftrightarrow$
(\ref{eq:FatouX2}).

In the case of $L^\infty$, the equivalence of (1) -- (3) is first
obtained by \citep{MR2011534} for \emph{sublinear expectation} (or
equivalently \emph{coherent} risk measures), i.e., for
\emph{positively homogeneous} monotone convex functions $\varphi$ with
(\ref{eq:CashAdd}). The case of convex risk functions is then proved
by \citep{jouini06:_law_fatou} with an additional assumption that
$L^1$ is separable, and the latter assumption is later removed by
\citep{MR2648597} by a homogenization trick.  As a crucial (but
trivial) feature of the space $L^\infty$, (\ref{eq:CashAdd}) and the
monotonicity already imply that $\varphi$ is finite everywhere. Thus
Theorem~\ref{thm:JSTGeneral2} slightly generalizes the previous
results just mentioned, and this generalization is crucial in the
implication (3) $\Rightarrow$ (2) (while not essential for other
implication), where a consequence of (\ref{eq:CashAdd}) that all the
level sets $\{Z\in L^1:\,\varphi^*(Z)\leq c\}$ are bounded in norm was
used to invoke (the proof of) James's sup-theorem.

Finally, we mention a remarkable fact, due to \citep{MR2509290}, that
the Lebesgue property on $L^\infty$ of a convex risk function is
necessary for the function to have a finite extension to some solid
rearrangement invariant space $\Xs\subset L^0$ properly containing
$L^\infty$.

\subsubsection{Orlicz Spaces and their Morse Subspaces (Orlicz hearts)}
\label{sec:Orlicz}

Let $\Phi:[0,\infty)\rightarrow [0,\infty]$ be a Young function
(increasing left-continuous convex function finite on a neighborhood
of $0$ with $\Phi(0)=0$ and $\lim_{x\rightarrow\infty}\Phi(x)=\infty$)
and let
\begin{align*}
  L^\Phi&:=\{X\in L^0:\, \exists
  \lambda>0,\,\EB[\Phi(\lambda|X|)]<\infty\}
  \quad \text{(Orlicz space)},\\
  M^\Phi&:=\{X\in L^0:\,\forall \lambda>0,\,
  \EB[\Phi(\lambda|X|)]<\infty\}\quad \text{(Morse subspace)}.
\end{align*}
Then both $L^\Phi$ and $M^\Phi$ are solid spaces, $L^\infty\subset
L^\Phi\subset L^1$ and $(L^\Phi)^\sim_n=L^{\Phi^*}$ in general, where
$\Phi^*(y):=\sup_{x\geq 0}(xy-\Phi(x))$ is the conjugate Young
function. 

If $\Phi$ is finite-valued (otherwise $M^\Phi=\{0\}$)), then $M^\Phi$
contains $L^\infty$ and $(M^\Phi)^\sim_n=(M^\Phi)^*=L^{\Phi^*}$ as
well. Thus in this case,
$\sigma(M^\Phi,(M^\Phi)^\sim_n)=\sigma(M^\Phi,L^{\Phi^*})$ is the weak
topology, and consequently (\ref{eq:FatouEquality1}) for
$\Xs=M^{\Phi}$ is equivalent to the norm-lower semicontinuity.  Any
norm convergent sequence in $M^\Phi$ admits a subsequence which is
dominated by some element of $M^\Phi$ and a.s. convergent to the same
limit, which shows the implication (\ref{eq:FatouX2}) $\Rightarrow$
(\ref{eq:FatouEquality1}). In fact, the Luxemburg norm
$\|\cdot\|_\Phi$ (see \citep{rao_ren91} for definition) is
order-continuous on $M^\Phi$ (though not on $L^\Phi$). Consequently,
the extended Namioka-Klee theorem tells us that any finite monotone
convex function $\varphi$ on $M^\Phi$ automatically has the Lebesgue
property. The equivalence of finiteness and (2), (3) in this case
(with (\ref{eq:CashAdd})) is contained in \citep{MR2509268} in a
slightly different form, where it is also shown that the finiteness of
a convex risk function is equivalent to $\mathrm{int}\,\dom
\varphi\neq \emptyset$. See also \citep{MR2507760} for the case
$\Xs=L^p=M^{\Phi_p}$ ($\Phi(x)=x^p/p$, $1\leq p<\infty$).

Recall that the Young function $\Phi$ is said to satisfy the
$\Delta_2$-condition if there exist constants $C>0$ and $x_0>0$ such
that $\Phi(2x)\leq C\Phi(x)$ for $x\geq x_0$.  In this case,
$L^\Phi=M^\Phi$ (the converse is also true if $(\Omega,\FC,\PB)$ is
atomless). Consequently,
$(L^\Phi)^\sim_n=(L^\Phi)^*
=L^{\Phi^*}$, thus the
argument of previous paragraph still applies.

When $\Xs=L^\Phi$ with $\Phi^*$ finite,
\citep{orihuela_ruiz12:_lebes_orlic} recently obtained the same
equivalence \emph{but under a stronger a priori assumption that
  $\varphi$ is $\sigma(L^\Phi,M^{\Phi^*})$-lower
  semicontinuous}. Though the latter assumption is not so
unreasonable,
it excludes some trivial cases unless $L^{\Phi^*}=M^{\Phi^*}$: For any
$Z\in L^{\Phi^*}_+\setminus M^{\Phi^*}$, $\varphi(X)=\EB[XZ]$ is a
monotone convex $\sigma(L^\Phi,L^{\Phi^*})$-lsc function on $L^\Phi$,
but is not $\sigma(L^\Phi,M^{\Phi^*})$-continuous (thus not lsc by
linearity) because
$(L^\Phi,\sigma(L^\Phi,M^{\Phi^*}))^*=M^{\Phi^*}$. Note that this
$\varphi$ has the Lebesgue property; if $|X_n|\leq Y\in L^\Phi$ and
$X_n\rightarrow X$ a.s., the dominated convergence theorem applies to
the sequence $(X_nZ)_n$ dominated by $YZ\in L^1$.  Recall that (with
$\Phi^*$ finite), $(M^{\Phi^*})^*=L^\Phi$, hence
$\sigma(M^{\Phi^*},L^\Phi)=\sigma(M^{\Phi^*},(M^{\Phi^*})^*)$
(\emph{the} weak topology). This was needed in
\citep{orihuela_ruiz12:_lebes_orlic} to invoke their \emph{perturbed
  James's sup-theorem} in the proof of (3) $\Rightarrow$ (2). In fact,
our proof of Theorem~\ref{thm:JSTGeneral2} use a similar version of
James's theorem \emph{but not to $\Xs^\sim_n$}. We reduce the problem
to $\sigma(L^1,L^\infty)$ by a simple ``change of variable'' trick,
which allows us to prove the implication for arbitrary solid space
$\Xs$ containing $L^\infty$ and contained in $L^1(\QB)$ with some
$\QB\sim\PB$ under the weaker assumption
(\ref{eq:FatouEquality1}). The equivalence of
(\ref{eq:FatouEquality1}) and the Fatou property (\ref{eq:FatouX2}) in
this level of generality is still open (to us).

\section{Proof of Theorem~\ref{thm:JSTGeneral2}}
\label{sec:ProofMain}

\emph{In the sequel, all the assumptions of
  Theorem~\ref{thm:JSTGeneral2} are in force without further
  mentioning.}  Recall that $\Xs$ is an order-complete
(Dedekind-complete) Riesz space with the countable-sup property as an
ideal of $L^0$, whose order-continuous dual space is identified with
$\Xs^\sim_n$ given by (\ref{eq:OderContiDual}), and we have
$\Xs^\sim_n=\Xs^\sim_c$ (the $\sigma$-order-continuous dual) in the
notation of \citep{aliprantis_Burkinshaw03} as a consequence of the
countable-sup property. We shall use the following characterization
of $\sigma(\Xs^\sim_n,\Xs)$-compact sets.

\begin{lemma}
  \label{lem:Compact1}
  Let $C\subset \Xs^\sim_n$. Then the following are equivalent:
  \begin{enumerate}
  \item $C$ is $\sigma(\Xs^\sim_n,\Xs)$-relatively compact,
  \item $X_n\downarrow 0$ in $\Xs$ implies $\sup_{Z\in
      C}\EB[X_n|Z|]\downarrow 0$.
  \item $\rho_C(X):= \sup_{Z\in C}\EB[|XZ|]$ is order-continuous,
    i.e., $X_n\in Y\in\Xs$, $|X_n|\leq |Y|$ a.s. for all $n$ and
    $X_n\rightarrow X$ a.s. imply $\rho_C(X)=\lim\rho_C(X_n)$.
  \item the set $\{XZ:\, Z\in C\}$ is uniformly integrable for all
    $X\in \Xs$.
  \end{enumerate}

\end{lemma}
\begin{proof}
  Since $\Xs$ is order-complete, $C\subset \Xs^\sim_n=\Xs^\sim_c$ is
  $\sigma(\Xs^\sim_n,\Xs)$-relatively compact if and only if the
  convex solid hull of $C$ is relatively compact for the same topology
  by \citep[][Corollary~6.29]{aliprantis_Burkinshaw03}, and the latter
  is equivalent to (2) by
  \citep[][Theorem~6.21]{aliprantis_Burkinshaw03}.  To prove (2)
  $\Rightarrow$ (3), let a sequence $(X_n)_n$ in $\Xs$ be dominated by
  $Y\in \Xs$ and converge a.s. to $X$. Then $\bar X_n:=\sup_{l\geq
    n}|X-X_l|\leq 2Y$, hence $\bar X_n\in \Xs$, $\bar X_n\downarrow
  0$, thus $|\rho_C(X)-\rho_C(X_n)|\leq \sup_{Z\in C}\EB[\bar
  X_n|Z|]\downarrow 0$.  For (3) $\Rightarrow$ (4), pick an arbitrary
  $X\in\Xs$. Then for any sequence $(A_n)\subset \FC$ with
  $\PB(A_n)\rightarrow 0$, we have $|X\ind_{A_n}|\leq |X|\in\Xs$
  and $X\ind_{A_n}\rightarrow 0$ a.s. Thus (3) shows that
  $\sup_{Z\in C}\EB[|XZ|\ind_{A_n}]=\sup_{Z\in
    C}\EB[|X\ind_{A_n}||Z|]\downarrow 0$, hence $\{XZ:\, Z\in C\}$
  is uniformly integrable for all $X\in \Xs$.

  (4) $\Rightarrow$ (2). Take $\QB\sim\PB$ such that $\Xs\subset
  L^1(\QB)$ and put $\Zh:=d\QB/d\PB$. Then note that $\PB(\Zh>0)=1$
  and $\sup_{Z\in C}\QB(Z/\Zh>N)\leq \sup_{Z\in
    C}\EB[|Z|]/N\rightarrow 0$ ($N\rightarrow\infty$) since $C$ is
  bounded in $L^1$ by taking $X=1\in\Xs$ in (4).  Since the one-point
  set $\{1/\Zh\}=\{d\PB/d\QB\}$ is $\QB$-uniformly integrable, we
  deduce that
  \begin{equation}
    \label{eq:ProofLemmaUI1}
    \sup_{Z\in C}\PB(Z/\Zh>N)=\sup_{Z\in C}\EB_\QB[(1/\Zh)\ind_{\{Z/\Zh>N\}}]
    \rightarrow 0.  
  \end{equation}
  Now suppose $ X_n\downarrow 0$ in $\Xs$.  Then $\EB[X_n|Z|]\leq
  \EB[X_n|Z|\ind_{\{|Z/\Zh|>N\}}]+\EB[X_n|Z|\ind_{\{|Z/\Zh|\leq
    N\}}] \leq\linebreak \EB[X_1|Z|\ind_{\{|Z/\Zh|>N\}}]+ N\EB_\QB[X_n]$ for
  all $Z\in C$, hence
  \begin{align*}
    \sup_{Z\in C}\EB[X_n|Z|]
    &\leq \sup_{Z\in C}\EB[X_1|Z|\ind_{\{|Z/\Zh|>N\}}]+
    N\EB_\QB[X_n].
  \end{align*}
  The first term in the RHS tends to $0$ by the uniform integrability
  of $\{X_1Z:\, Z\in C\}$ and (\ref{eq:ProofLemmaUI1}), while for each
  $N$ fixed, the second term tends to $0$ as $n\rightarrow \infty$ by
  the dominated convergence theorem since $0\leq X_n\leq X_1\in
  L^1(\QB)$. 
\end{proof}

The next preparatory result is a version of James's theorem obtained
by \citep{orihuelaRuizGalan12:_james} which we shall use in the proof
of (3) $\Rightarrow$ (2).
\begin{theorem}[\citep{orihuelaRuizGalan12:_james}, Theorem~2]
  \label{thm:James}
  Let $E$ be a real Banach space and
  $f:E\rightarrow\RB\cup\{+\infty\}$ be a function which is not
  identically $+\infty$ and is coercive, i.e.,
  \begin{equation}
    \label{eq:Coercive}
    \lim_{\|x\|\rightarrow\infty}\frac{f(x)}{\|x\|}=+\infty.
  \end{equation}
  Then if the supremum $\sup_{x\in E}(\langle x,x^*\rangle-f(x))$ is
  attained for every $x^*\in E^*$, the level set $\{x\in E:\, f(x)\leq
  c\}$ is relatively weakly compact for each $c\in\RB$.
\end{theorem}

\begin{lemma}
  \label{lem:Estim1}
  Suppose $\varphi(0)=0$, then for any $\beta\in\RB$, $X\in \Xs$ and
  $Z\in \Xs^\sim_n$,
  \begin{equation}
    \label{eq:Estim1}
    \EB[XZ]-\varphi^*(Z)\geq-\beta\,\Rightarrow\, \varphi^*(Z)\leq 2\beta+2\varphi(2|X|).
  \end{equation}

\end{lemma}
\begin{proof}
  Since $0=\varphi(0)=\inf_{Z\in\Xs^\sim_n}\varphi^*(Z)$, there
  exists, for any $\varepsilon>0$, some
  $Z_\varepsilon\in(\Xs^\sim_{n})_+$ with
  $\varphi^*(Z_\varepsilon)<\varepsilon$ (by the monotonicity of
  $\varphi$, $\dom\varphi^*\subset (\Xs^\sim_n)_+$). Thus
  \begin{align*}
    \EB[XZ]\leq \EB\left[2|X|\frac{Z+Z_\varepsilon}2\right]\leq
    \varphi(2|X|)+\frac{\varphi^*(Z)+\varphi^*(Z_\varepsilon)}2 \leq
    \varphi(2|X|)+\frac{\varphi^*(Z)}{2}+\varepsilon
  \end{align*}
  for any $\varepsilon$, hence $\EB[XZ]\leq
  \varphi(2|X|)+\varphi^*(Z)/2$. Consequently,
  $\EB[XZ]-\varphi^*(Z)\geq -\beta$ implies $\varphi^*(Z)\leq
  \beta+\EB[XZ]\leq \beta+\varphi(2|X|)+\varphi^*(Z)/2$. Rearranging
  the terms, we have $\varphi^*(Z)\leq 2\beta+2\varphi(2|X|)$.
\end{proof}

\begin{proof}[Proof of Theorem~\ref{thm:JSTGeneral2}] %
  We suppose without loss of generality that $\varphi(0)=0$ and we
  write $\Lambda_c:=\{Z\in \Xs^\sim_n:\, \varphi^*(Z)\leq c\}$
  throughout this proof. Note that $\Lambda_c$ is
  $\sigma(\Xs^\sim_n,\Xs)$-closed for each $c>0$ since $\varphi^*$ is
  lower semicontinuous for the same topology. Also, the monotonicity
  of $\varphi$ implies that $\Lambda_c\subset \dom\varphi^*\subset
  L^1_+$.

  (1) $\Rightarrow$ (2). By the above comments, it suffices to show
  that for each $c>0$, $\Lambda_c$ is
  $\sigma(\Xs^\sim_n,\Xs)$-\emph{relatively} compact, which is
  equivalent (in view of Lemma~\ref{lem:Compact1}) to saying
  that $\sup_{Z\in\Lambda_c}\EB[X_n|Z|]\linebreak\downarrow 0$ if $\Xs\ni
  X_n\downarrow 0$. Given such a sequence $(X_n)_n$, observe
  that  $\sup_{Z\in \Lambda_c}\EB[X_nZ]\leq\linebreak\frac1\lambda
  \sup_{Z\in \Lambda_c}\left(\varphi(\lambda
    X_n)+\varphi^*(Z)\right)\leq \frac1\lambda \varphi(\lambda
  X_n)+\frac{c}\lambda$ by Young's inequality. Then the Lebesgue
  property of $\varphi$ implies that $\varphi(\lambda X_n)\downarrow
  0$ for every $\lambda>0$ since $X_n\downarrow 0$ (so $\lambda
  X_n\downarrow 0$). Hence a diagonal argument yields the desired
  property.

  (2) $\Rightarrow$ (3).  Note that $Z\mapsto \EB[XZ]-\varphi^*(Z)$ is
  $\sigma(\Xs^\sim_n,\Xs)$-upper semicontinuous (since $\varphi^*$ is
  lower semicontinuous). Thus, the set $\Gamma_X:=\{Z\in\Xs^\sim_n:\,
  \EB[XZ]-\varphi^*(Z)\geq \varphi(X)-1\}$ is
  $\sigma(\Xs^\sim_n,\Xs)$-closed, and it is contained in the
  $\sigma(\Xs^\sim_n,\Xs)$-compact set
  $\Lambda_{2-2\varphi(X)+2\varphi(2|X|)}$ by Lemma~\ref{lem:Estim1},
  hence $\Gamma_X$ itself is $\sigma(\Xs^\sim_n,\Xs)$-compact. Now (3)
  is clear since any upper semicontinuous function on a compact set
  attains its maximum.

  (2) $\Rightarrow$ (1).  Let $(X_n)_n$ be such that $|X_n|\leq Y\in
  \Xs_+$ and $X_n\rightarrow X$ a.s., then $|X|\leq Y$ as well, and
  $\varphi(X_n), \varphi(X)\geq \varphi(-Y)$ by the monotonicity. Thus
  given (\ref{eq:FatouEquality1}), only those $Z\in\Xs^\sim_n$ with
  $\EB[X_nZ]-\varphi^*(Z)\geq \varphi(-Y)-1$ contribute to the
  supremum $\sup_{Z\in\Xs^\sim_n}(\EB[XZ]-\varphi^*(Z))$ (and
  the same is true for $X$). Applying Lemma~\ref{lem:Estim1} and using
  the notation of the previous paragraph, any such $Z\in\Xs^\sim_n$ is
  contained in $\Lambda_{c_Y}$ where
  $c_Y:=2-2\varphi(-Y)+2\varphi(2Y)$. Consequently,
  $\varphi(X_n)=\sup_{Z\in\Lambda_{c_Y}}(\EB[X_nZ]-\varphi^*(Z))$,
  $\varphi(X)=\sup_{Z\in\Lambda_{c_Y}}(\EB[XZ]-\varphi^*(Z))$, thus
  applying twice the elementary inequality $\sup_{x\in
    A}f(x)-\sup_{x\in A}g(x)\leq \sup_{x\in A}(f(x)-g(x))$
  \citep[e.g.,][p.~356, Prop.~12]{MR979294},
  \begin{align*}
    |\varphi(X)-\varphi(X_n)|
    &\leq \sup_{Z\in\Lambda_{c_Y}}\left|\EB[XZ]-\varphi^*(Z)-(\EB[X_nZ]-\varphi^*(Z))\right|\\
    &\leq \sup_{Z\in\Lambda_{c_Y}}\EB[|X_n-X|Z]\rightarrow 0.
  \end{align*}
  by the $\sigma(\Xs^\sim,\Xs)$-compactness of $\Lambda_{c_Y}$ and
  Lemma~\ref{lem:Compact1}.

  (3) $\Rightarrow$ (2). We shall prove that for any $X\in\Xs$ and
  $c>0$, the set $\{(1+|X|)Z\in L^1:\,
  Z\in\Xs^\sim_n,\,\varphi^*(Z)\leq c\}$ is uniformly integrable,
  which in view of Lemma~\ref{lem:Compact1} shows the desired
  compactness. Let us define a function $g_X:L^1\rightarrow
  \RB\cup\{+\infty\}$ by
  \begin{align*}
    g_X(Z):=
    \begin{cases} \varphi^*\left(\frac{Z}{1+|X|}\right)
      &\text{ if }\frac{Z}{1+|X|}\in \Xs^\sim_n,\\
      +\infty&\text{ otherwise},
    \end{cases}
    \quad\forall Z\in L^1.
  \end{align*}
  Then for any $c>0$, $g_X(Z)\leq c$ iff $Z=(1+|X|)Z'$,
  $Z'\in\Xs^\sim_n$ and $\varphi^*(Z')\leq c$, and since $\Xs$ is an
  ideal containing $L^\infty$, $\lambda(1+|X|)\mathrm{sgn}(Z)\in
  \Xs$. Hence for any $Z\in\dom g_X$,
  \begin{align*}
    g_X(Z)&\geq
    \EB\left[\frac{\lambda(1+|X|)\mathrm{sgn}(Z)}{1+|X|}Z\right]-\varphi(\lambda(1+|X|)\mathrm{sgn}(Z))\\
    &\geq \lambda \|Z\|_{L^1}-\varphi(-\lambda(1+|X|)).
  \end{align*}
  Since $\varphi$ is finite on $\Xs$, this shows that $g_X$ is a
  coercive function on $L^1$.

  Let $X\in \Xs$, $Y\in L^\infty$ and $Z_{X,Y}\in \Xs^\sim_n$ be a
  maximizer in (3) for $Y(1+|X|)$ which belongs to $\Xs$ by the
  solidness (since $Y(1+|X|)\leq \|Y\|_\infty(1+|X|)\in\Xs$). Then
  \begin{align*}
    \varphi(Y(1+|X|))&=
    \EB[Y(1+|X|)Z_{X,Y}]-\varphi^*(Z_{X,Y})\\
    &=\EB[Y(1+|X|)Z_{X,Y}]-g_X((1+|X|)Z_{X,Y}).
  \end{align*}
 On the other hand, for $Z'\in L^1$, if
  $Z'\in\dom g_X$,
  \begin{align*}
    \EB[YZ']-g_X(Z')=\EB\left[Y(1+|X|)\frac{Z'}{1+|X|}\right]-\varphi^*\left(\frac{Z'}{1+|X|}\right)\leq
    \varphi((1+|X|)Y),
  \end{align*}
  while if $Z'\in L^1\setminus \dom g_X$, then obviously
  $\EB[YZ']-g_X(Z')=-\infty\leq \varphi((1+|X|)Y)$. We have thus shown
  that the supremum $\sup_{Z\in L^1}(\EB[YZ]-g_X(Z))$ is attained by
  $(1+|X|)Z_{X,Y}$ for all $Y\in L^\infty$. Consequently,
  Theorem~\ref{thm:James} shows that $\{Z'\in L^1:\, g_X(Z')\leq
  c\}=\{(1+|X|)Z:\, Z\in \Xs^\sim_n,\,\varphi^*(Z)\leq c\}$ is
  $\sigma(L^1,L^\infty)$-relatively compact, hence uniformly
  integrable by the Dunford-Pettis theorem. 
\end{proof}

\section*{Acknowledgements}

The author thanks Sara Biagini for helpful comments on an earlier
version of the paper. He also thanks two anonymous referees for
careful reading. The financial support of the Center for Advanced
Research in Finance (CARF) at the Graduate School of Economics of the
University of Tokyo is gratefully acknowledged.

\bibliographystyle{OwariRev}
\bibliography{main}
\end{document}